\documentclass{article}
\usepackage{amsmath,amssymb,amsthm,amscd,dsfont}
\numberwithin{equation}{section} \allowdisplaybreaks
\begin{document}
\newtheorem{theorem}{Theorem}[section]
\newtheorem{defin}{Definition}[section]
\newtheorem{prop}{Proposition}[section]
\newtheorem{corol}{Corollary}[section]
\newtheorem{lemma}{Lemma}[section]
\newtheorem{rem}{Remark}[section]
\newtheorem{example}{Example}[section]
\title{Hessian geometry on Lagrange spaces}
\author{{\small by}\vspace{2mm}\\Izu Vaisman}
\date{}
\maketitle
{\def\thefootnote{*}\footnotetext[1]%
{{\it 2000 Mathematics Subject Classification: 53C15, 53C60.}
}}
\begin{center} \begin{minipage}{10cm}
A{\footnotesize BSTRACT. We extend the correspondence between Hessian and K\"ahler metrics and curvatures to Lagrange spaces.}
\end{minipage}
\end{center} \vspace*{5mm}
\noindent
Hessian geometry on locally affine manifolds was studied by several authors, particularly, Cheng and Yau \cite{CY} and H. Shima \cite{S}. Shima introduced a notion of Hessian curvature, which is a finer invariant than Riemannian curvature\footnote{The Riemannian curvature of metrics defined by the Hessian of a function was studied extensively, e.g., \cite{To}.} and is related with the curvature of an associated K\"ahler metric on the tangent manifold (the total space of the tangent bundle). A Lagrange space is a manifold with a regular Lagrangian, also called a Lagrange metric, on its tangent manifold \cite{BM}. The latter has the vertical foliation by fibers and the fiber-wise Hessian of the Lagrangian defines (pseudo) Hessian\footnote{``Pseudo" is added if the metric is not positive definite.} metrics of the fibers. In this note, we extend the correspondence Hessian versus K\"ahler to the vertical foliation of the tangent manifold of a Lagrange space (Section 2). The subject of the note is not Lagrangian dynamics but, Hessian geometry and curvature in the context of
Lagrange spaces, which are a generalization of (pseudo) Finsler spaces. The study of curvature is motivated by the general principle that curvature invariants differentiate between spaces of a given type. We will begin by recalling the basics of Hessian and tangent bundle geometry (Section 1), since the reader is not supposed to be an expert on any of these, and by some required preparations. In an appendix we give index-free proofs of some properties of Hessian curvature established via local coordinates in \cite{S}. We work in the $C^\infty$ category and use the standard notation of differential geometry \cite{KN}.
\section{Preliminaries}
This is a preliminary section where we recall Hessian metrics and curvature and the basics of the geometry of tangent bundles. We refer to \cite{S} for Hessian geometry and to \cite{{BM},{VL}} for the tangent bundle geometry.
\subsection{Hessian geometry}
Let $N$ be a locally affine manifold with the flat, torsionless connection $\nabla^0$. A {\it (pseudo) Hessian metric (structure)} on $N$ is a (pseudo) Riemannian metric $g$ such that
\begin{equation}\label{ghess0} g|_{U_\alpha}(\mathcal{Y},\mathcal{Y}')=\nabla^0_{\mathcal{Y}'} \nabla^0_{\mathcal{Y}}\varphi_\alpha,
\end{equation}
where $\{U_\alpha\}$ is an open covering of $N$, $(\mathcal{Y},\mathcal{Y}')$ are local, parallel vector fields and $\varphi_\alpha\in C^\infty(U_\alpha)$.
If (\ref{ghess0}) holds on $N$ with a function $\varphi\in C^\infty(N)$, the metric is {\it globally (pseudo) Hessian}. Since local parallel vector fields are of the form\footnote{In the paper we use the Einstein summation convention.} $\mathcal{Y}=c^u(\partial/\partial y^u)$, where $(y^u)$ are local affine coordinates and $c^u=const.$, (\ref{ghess0}) is equivalent to
\begin{equation}\label{coefg} g|_{U_\alpha}=g_{uv} dy^u\otimes
dy^v,\;\;\;	 g_{uv}=\frac{\partial^2\varphi_\alpha}{\partial y^u\partial y^v}.\end{equation}

Let $\gamma$ be an arbitrary (pseudo) Riemannian metric on $N$. The formula
\begin{equation}\label{tCartan} C(\mathcal{Y},\mathcal{Y}',\mathcal{Y}'')= (\nabla^0_{\mathcal{Y}}\gamma)(\mathcal{Y}',\mathcal{Y}'')\end{equation}
defines a tensor, which we call the {\it Cartan tensor}. If the arguments are parallel vector fields, in particular vectors $\partial/\partial y^u$, the result is
\begin{equation}\label{tCartan1} C(\mathcal{Y},\mathcal{Y}',\mathcal{Y}'')= \mathcal{Y}(\gamma(\mathcal{Y}',\mathcal{Y}'')),\;
C_{uvw}=\frac{\partial g_{vw}}{\partial y^u}.\end{equation} The latest formula shows that $\gamma$ is a (pseudo) Hessian metric with components as in (\ref{coefg}) iff the tensor $C$ is totally symmetric.

The following question is natural: what are the conditions that characterize the class of (pseudo) Hessian manifolds $(N,g)$ within the class of (pseudo) Riemannian manifolds $(M,\gamma)$? The most straightforward answer\footnote{A significant answer to the question was given in \cite{Du}.} is that a (pseudo) Riemannian manifold is (pseudo) Hessian iff: 1) the Levi-Civita connection $\nabla$ of $\gamma$ can be deformed into a torsion-less flat connection $\nabla^0$ and 2) the Cartan tensor of the resulting pair $(\gamma,\nabla^0)$ is symmetric.

This remark motivates the introduction of the {\it difference (deformation) tensor} $\Phi=\nabla-\nabla^0$ \cite{S}, which has the following obvious properties
\begin{equation}\label{propPhi}
\Phi(\mathcal{Y},\mathcal{Y}')=\nabla_{\mathcal{Y}}\mathcal{Y}',
\;\Phi(\mathcal{Y},\mathcal{Y}')=\Phi(\mathcal{Y}',\mathcal{Y}),
\end{equation} where $\mathcal{Y}'$ is parallel in the first equality. The second equality is a consequence of the first since two parallel vector fields commute and $\nabla$ has no torsion.
The following lemma computes the difference tensor in the (pseudo) Hessian case.
\begin{lemma}\label{PhicuC} If the metric $\gamma$ is (pseudo) Hessian, then,
\begin{equation}\label{PhiC}
\gamma(\mathcal{Y}'',\Phi(\mathcal{Y},\mathcal{Y}')) =\frac{1}{2}C(\mathcal{Y}'',\mathcal{Y},\mathcal{Y}').
\end{equation}
\end{lemma}
\begin{proof} Since $C$ is a tensor, it suffices to evaluate $C$ on parallel vector fields, which we shall assume for all the arguments below, hence, (\ref{tCartan1}) holds.
On the other hand, from the well known global expression of the Levi-Civita connection (\cite{KN}, vol. I, \&IV.2), and since the bracket of two parallel vector fields vanishes, we have
$$\begin{array}{l}
\gamma(\mathcal{Y}'',\Phi(\mathcal{Y},\mathcal{Y}'))=
\gamma(\mathcal{Y}'',D_{\mathcal{Y}}\mathcal{Y}')= \frac{1}{2}\{\mathcal{Y}(\gamma(\mathcal{Y}',\mathcal{Y}'')) +\mathcal{Y}'(\gamma(\mathcal{Y},\mathcal{Y}''))\vspace*{2mm}\\ -
\mathcal{Y}''(\gamma(\mathcal{Y},\mathcal{Y}'))\}
=\frac{1}{2}\{C^{S'}(\mathcal{Y},\mathcal{Y}',\mathcal{Y}'') +C^{S'}(\mathcal{Y}',\mathcal{Y},\mathcal{Y}'') -C^{S'}(\mathcal{Y}'',\mathcal{Y},\mathcal{Y}')\}. \end{array}$$
If $\gamma$ is (pseudo) Hessian, $C$ is symmetric and we get the required result.\end{proof}

Condition 1) requires the relation between the curvatures of $\nabla,\nabla^0$. For arbitrary connections, a technical calculation that starts with the definitions gives the known formulas
\begin{equation}\label{curbDB}\begin{array}{l} T_{\nabla}(\mathcal{Y},\mathcal{Y}')=T_{\nabla^0}(\mathcal{Y},\mathcal{Y}') +\Phi(\mathcal{Y},\mathcal{Y}')-\Phi(\mathcal{Y}',\mathcal{Y}),\vspace*{2mm}\\
R_{\nabla}(\mathcal{Y},\mathcal{Y}')\mathcal{Y}'' -\Phi(T_{\nabla}(\mathcal{Y},\mathcal{Y}'),\mathcal{Y}'')\vspace*{2mm}\\ = R_{\nabla^0}(\mathcal{Y},\mathcal{Y}')\mathcal{Y}''
+\mathfrak{Q}(\mathcal{Y}',\mathcal{Y}'')\mathcal{Y}
-\mathfrak{Q}(\mathcal{Y},\mathcal{Y}'')\mathcal{Y}',
\end{array}
\end{equation}
where the arguments are arbitrary tangent vectors of $N$, $R$ denotes curvature, $T$ denotes torsion and
\begin{equation}\label{Xi} \mathfrak{Q}(\mathcal{Y}',\mathcal{Y}'')\mathcal{Y}=
\nabla^0_{\mathcal{Y}}(\Phi(\mathcal{Y}',\mathcal{Y}''))  -\Phi(\nabla_{\mathcal{Y}}\mathcal{Y}',\mathcal{Y}'') -\Phi(\mathcal{Y}',\nabla_{\mathcal{Y}}\mathcal{Y}'')
\end{equation}
is a {\it mixed covariant derivative}. As a consequence, we get the following necessary condition for a (pseudo) Hessian metric.
\begin{prop}\label{condnec} If $\gamma$ is a (pseudo) Hessian metric, there exists a symmetric deformation tensor $\Phi$ such that the Riemannian curvature of $\gamma$ satisfies the relation
\begin{equation}\label{RsiQgotic} R_{\nabla}(\mathcal{Y}',\mathcal{Y}, \mathcal{Y}_1,\mathcal{Y}_2)= \mathfrak{Q}(\mathcal{Y}',\mathcal{Y}_1, \mathcal{Y}_2,\mathcal{Y})- \mathfrak{Q}(\mathcal{Y}',\mathcal{Y}_2, \mathcal{Y}_1,\mathcal{Y}),\end{equation}
where
$$\begin{array}{l}
R_\nabla(\mathcal{Y}_1,\mathcal{Y}_2,\mathcal{Y}_3,\mathcal{Y}_4) =\gamma(\mathcal{Y}_1, R_\nabla(\mathcal{Y}_3,\mathcal{Y}_4)\mathcal{Y}_2),\vspace*{2mm}\\	 \mathfrak{Q}(\mathcal{Y}_1,\mathcal{Y}_2,\mathcal{Y}_3,\mathcal{Y}_4) =\gamma(\mathcal{Y}_1, \mathfrak{Q}(\mathcal{Y}_3,\mathcal{Y}_4)\mathcal{Y}_2).\end{array}$$
\end{prop}
\begin{proof} Use $T_\nabla=0$ and use parallel arguments $\mathcal{Y}$ in (\ref{curbDB}).
\end{proof}

We shall return to the terminology of \cite{S} as follows.
\begin{defin}\label{defHesscurb} {\rm
The {\it Hessian curvature operator and tensor} are, respectively, given by\footnote{In the covariant curvature tensor $Q$ we have preferred an order of arguments which is different from the order used in \cite{S}.}
\begin{equation}\label{eqHessc} \begin{array}{l}
Q(\mathcal{Y}_1,\mathcal{Y}_2)\mathcal{Y} =
\nabla^0_{\mathcal{Y}}(\Phi(\mathcal{Y}_1,\mathcal{Y}_2))  -\Phi(\nabla^0_{\mathcal{Y}}\mathcal{Y}_1,\mathcal{Y}_2) -\Phi(\mathcal{Y}_1,\nabla^0_{\mathcal{Y}}\mathcal{Y}_2),\vspace*{2mm}\\
Q(\mathcal{Y},\mathcal{Y}',\mathcal{Y}_1,\mathcal{Y}_2)=
\gamma(Q(\mathcal{Y}_1,\mathcal{Y}_2)\mathcal{Y}',\mathcal{Y}).
\end{array}\end{equation}}\end{defin}
\begin{prop}\label{lemaQQ}
The Hessian curvature tensor of a (pseudo) Hessian metric $\gamma$ is related to the mixed covariant derivative $\mathfrak{Q}$ by the relations
\begin{equation}\label{exprQgotic} \begin{array}{l}
Q(\mathcal{Y},\mathcal{Y}',\mathcal{Y}_1,\mathcal{Y}_2)=
\mathfrak{Q}(\mathcal{Y},\mathcal{Y}',\mathcal{Y}_1,\mathcal{Y}_2)
\vspace*{2mm}\\ +\frac{1}{2}[C(\mathcal{Y},\nabla_{\mathcal{Y}'}\mathcal{Y}_1, \mathcal{Y}_2)- C(\mathcal{Y}, \mathcal{Y}_1,\nabla_{\mathcal{Y}'}\mathcal{Y}_2)]\vspace*{2mm}\\
=\mathfrak{Q}(\mathcal{Y},\mathcal{Y}',\mathcal{Y}_1,\mathcal{Y}_2)
 +[\gamma(\nabla_{\mathcal{Y}'}\mathcal{Y}_1, \nabla_{\mathcal{Y}}\mathcal{Y}_2) +\gamma(\nabla_{\mathcal{Y}'}\mathcal{Y}_2, \nabla_{\mathcal{Y}}\mathcal{Y}_1)], \end{array}\end{equation}
where the arguments $\mathcal{Y}$ are parallel vector fields.\end{prop}
\begin{proof} The difference $Q-\mathfrak{Q}$ is provided by (\ref{Xi}), (\ref{eqHessc}). Then, using the first equality (\ref{propPhi}), (\ref{PhiC}), the symmetry of $C$ and the preservation of $\gamma$ by $\nabla$ we get the required formula.\end{proof}

Formulas (\ref{exprQgotic}) lead to the following reformulation of the necessary condition given in Proposition \ref{condnec} for (pseudo) Hessian metrics.
\begin{prop}\label{proprelRQ} {\rm\cite{S}} If $\gamma$ is a (pseudo) Hessian metric, the following relation between the Riemannian and the Hessian curvature of $\gamma$ holds:
\begin{equation}\label{relRQ} R_{\nabla}(\mathcal{Y}',\mathcal{Y},\mathcal{Y}_1,\mathcal{Y}_2)
=\frac{1}{2}[Q(\mathcal{Y}',\mathcal{Y}_1,\mathcal{Y}_2,\mathcal{Y})
-Q(\mathcal{Y}',\mathcal{Y}_2,\mathcal{Y}_1,\mathcal{Y})].\end{equation}
\end{prop}
\begin{proof} We may assume that the arguments are parallel vector fields. Then, (\ref{RsiQgotic}), (\ref{exprQgotic}) imply
$$\begin{array}{l}
R_{\nabla}(\mathcal{Y}',\mathcal{Y},\mathcal{Y}_1,\mathcal{Y}_2)=
Q(\mathcal{Y}',\mathcal{Y}_1,\mathcal{Y}_2,\mathcal{Y})
-Q(\mathcal{Y}',\mathcal{Y}_2,\mathcal{Y}_1,\mathcal{Y}) \vspace*{2mm}\\ \hspace*{2mm} -\gamma(\nabla_{\mathcal{Y}_1}\mathcal{Y},\nabla_{\mathcal{Y}'}\mathcal{Y}_2)+
\gamma(\nabla_{\mathcal{Y}_2}\mathcal{Y},\nabla_{\mathcal{Y}'}\mathcal{Y}_1) \vspace*{2mm}\\ \hspace*{2mm}=
Q(\mathcal{Y}',\mathcal{Y}_1,\mathcal{Y}_2,\mathcal{Y})
-Q(\mathcal{Y}',\mathcal{Y}_2,\mathcal{Y}_1,\mathcal{Y}) -
\gamma(R_\nabla(\mathcal{Y}_1,\mathcal{Y}_2)\mathcal{Y},\mathcal{Y}').
\end{array}$$
(We have used the commutation of parallel vector fields and the properties $\nabla\gamma=0$, $T_\nabla=0$.) This is equivalent to (\ref{relRQ}).
\end{proof}
\begin{rem}\label{2raport} {\rm We refer the reader to \cite{S} for the expression of the local components of the Hessian curvature. See also the Appendix of the present paper for the evaluation of the Hessian curvature on parallel arguments.}\end{rem}
\subsection{Tangent bundle geometry}
Let $M$ be an $m$-dimensional manifold and $TM$ its tangent bundle with the total space $\mathcal{T}M$, called the {\it tangent manifold} of $M$.
The differentiable structure of $\mathcal{T}M$ is given by local coordinates
$(x^i,y^i)$ where $i=1,...,m$, $x^i$ are local coordinates on $M$ and $y^i$ are vector coordinates with respect to the basis $\partial/\partial x^i$. The corresponding coordinate transformations are:
\begin{equation}\label{coordtr} \tilde{x}^i=\tilde{x}^i(x^j),
\,\tilde{y}^i=\frac{\partial\tilde{x}^i}{\partial x^j}y^j.
\end{equation}
The fibers of $TM$ define the {\it vertical foliation} $\mathcal{V}$. We will use the same symbol $\mathcal{V}$ for the tangent bundle of the vertical leaves.

A tangent vector $X=\xi^i(\partial/\partial x^i)\in T_xM$ has a {\it vertical lift}
defined by
$$(\xi^i\frac{\partial}{\partial x^i})^v=\xi^i\frac{\partial}{\partial y^i}.$$
The formula $S\mathcal{X}=(\pi_*\mathcal{X})^v$ ($\mathcal{X}\in T(\mathcal{T}M),\pi:TM\rightarrow M$) defines a Nijenhuis tensor field\footnote{A Nijenhuis tensor field is an endomorphism of $TM$ that has a vanishing Nijenhuis tensor
$$\mathcal{N}_S(\mathcal{X},\mathcal{Y})= [S\mathcal{X},S\mathcal{Y}]-S[S\mathcal{X},\mathcal{Y}] -S[\mathcal{X},S\mathcal{Y}]+S^2[\mathcal{X},\mathcal{Y}].$$}
$S\in End(T\mathcal{T}M)$ with the properties $S^2=0,\,im\,S=ker\,S=\mathcal{V}$, called the {\it tangent structure} of $\mathcal{T}M$.

A {\it tangent metric} on $\mathcal{T}M$ is a (pseudo) Riemannian metric $\gamma$ with a non degenerate restriction to $\mathcal{V}$ and such that
$$ \gamma(S\mathcal{X},S\mathcal{Y})=\gamma(\mathcal{X},\mathcal{Y}),
\;\;\forall \mathcal{X},\mathcal{Y}\perp_\gamma\mathcal{V}.
$$
Then, $\mathcal{H}=\mathcal{V}^{\perp_\gamma}$ is a complement of $\mathcal{V}$ in $\mathcal{T}M$, $g=\gamma|_{\mathcal{H}}$ is a non degenerate metric on $\mathcal{H}$ and $S$ defines an isometry $(\mathcal{H},g)\approx(\mathcal{V},\gamma|_{\mathcal{V}})$ with the inverse $S'$, which we extend by $0$ on $\mathcal{H}$. The mapping $\gamma\mapsto(\mathcal{H},g)$ is a bijection between tangent metrics and pairs consisting of a complementary bundle and a transversal metric of the foliation $\mathcal{V}$. If we start with the pair $(\mathcal{H},g)$, $\gamma$ is defined by $\mathcal{H}\perp_\gamma\mathcal{V}$ and
$$ \gamma(\mathcal{X},\mathcal{Y})= \gamma(S\mathcal{X},S\mathcal{Y})=g(\mathcal{X},\mathcal{Y}),\;
\forall\mathcal{X},\mathcal{Y}\in\mathcal{H}.
$$

Any complement $\mathcal{H}$ of $\mathcal{V}$ ($\mathcal{H}\oplus\mathcal{V}=T\mathcal{T}M$) is called a {\it horizontal bundle} and also a {\it nonlinear connection}. A vector $X\in TM$ has a {\it horizontal lift} $X^h$ characterized by $X^h\in\mathcal{H},\pi_*X^h=X$. The horizontal lifts of $\partial/\partial x^i$ yield local tangent bases of $T\mathcal{T}M$, \begin{equation}\label{Hbases}X_i=\left(\frac{\partial}{\partial x^i}\right)^h= \frac{\partial}{\partial x^i}-t_i^j\frac{\partial}{\partial y^j},\;\frac{\partial}{\partial y^i},\end{equation} with the dual  cotangent bases $(dx^i,\,\theta^i=dy^i+t^i_jdx^j)$, where $t_i^j$ are local functions on $\mathcal{T}M$, known as the coefficients of the nonlinear connection $\mathcal{H}$. If a horizontal bundle was chosen, it is convenient to look at the transversal tensors of $\mathcal{V}$ ({\it horizontal tensors}) as tensors on $\mathcal{T}M$ by extending them by zero if evaluated on at least one vertical argument and, similarly, to extend $\mathcal{V}$-tensors ({\it vertical tensors}) by zero on a horizontal argument. On the other hand, we can {\it reflect} vertical tensors to horizontal tensors and vice-versa by first applying $S$, $S'$, respectively, to the arguments. The reflection of $\tau$ will be denoted by $\tau^{S'}$, if $\tau$ is horizontal, respectively, $\tau^{S}$, if $\tau$ is vertical.

Let us fix a decomposition $T\mathcal{T}M=\mathcal{H}\oplus\mathcal{V}$. A {\it Bott connection} is a linear connection on $\mathcal{T}M$ that preserves the subbundles $\mathcal{H},\mathcal{V}$ and satisfies the conditions
\begin{equation}\label{BottV} \nabla_{\mathcal{X}}\mathcal{Y} =pr_{\mathcal{V}}[\mathcal{X},\mathcal{Y}],\,
\nabla_{\mathcal{Y}}\mathcal{X} =pr_{\mathcal{H}}[\mathcal{Y},\mathcal{X}],
\;\mathcal{X}\in\mathcal{H},\mathcal{Y}\in\mathcal{V},\end{equation}
where ``pr" stands for ``projection".

On $\mathcal{T}M$ there exists a unique Bott connection that preserves the tensor fields $S,S'$; it is given by adding to (\ref{BottV})  the derivatives
\begin{equation}\label{Berwald} \nabla^B_{\mathcal{X}}\mathcal{X}'=S'pr_{\mathcal{V}}
[\mathcal{X},S\mathcal{X}'],\, \nabla^B_{\mathcal{Y}}\mathcal{Y}'=Spr_{\mathcal{H}}
[\mathcal{Y},S'\mathcal{Y}'],\end{equation} $\forall\mathcal{X},\mathcal{X}'\in
\mathcal{H},\mathcal{Y},\mathcal{Y}'\in
\mathcal{V}$. $\nabla^B$ is called the {\it Berwald connection} \cite{BM}.

If $D$ is an arbitrary linear connection, we get an {\it associated Bott connection}\footnote{In \cite{Bej0} $\nabla^D$ is called a {\it Vr\u anceanu connection} since the author traced back the history of this connection to a 1931 paper by G. Vr\u anceanu \cite{Vr}.} $\nabla^D$ given by
(\ref{BottV}) and
$$
\nabla^D_{\mathcal{X}}\mathcal{X}'
=pr_{\mathcal{H}}D_{\mathcal{X}}\mathcal{X}',\, \nabla^D_{\mathcal{Y}}\mathcal{Y}'
=pr_{\mathcal{V}}D_{\mathcal{Y}}\mathcal{Y}'\; (\mathcal{X},\mathcal{X}'\in\mathcal{H},\,
\mathcal{Y},\mathcal{Y}'\in\mathcal{V}).
$$

If $\gamma$ is a tangent metric such that $\mathcal{H}\perp_\gamma\mathcal{V}$, the Bott connection $\nabla^D$ associated to the Levi-Civita connection $D$ of $\gamma$ will be called the {\it canonical connection} of $\gamma$. It  is the unique Bott connection that satisfies the conditions \cite{{V71},{V73}}
\begin{equation}\label{propconexcan} \begin{array}{l}
\nabla_{\mathcal{X}}^D\gamma(\mathcal{Y},\mathcal{Z})=0,
\;{\rm for}\;\mathcal{X},\mathcal{Y},\mathcal{Z}\in\mathcal{H}\;{\rm and}\;\mathcal{X},\mathcal{Y},\mathcal{Z}\in\mathcal{V},\vspace*{2mm}\\
pr_{\mathcal{H}}T^{\nabla^D}(\mathcal{X},\mathcal{Y})=0\;{\rm if}\;\mathcal{X},\mathcal{Y}\in\mathcal{H},\;
pr_{\mathcal{V}}T^{\nabla^D}(\mathcal{X},\mathcal{Y})=0\;{\rm if}\;\mathcal{X},\mathcal{Y}\in\mathcal{V}.\end{array}\end{equation}
The restriction of the canonical connection to the vertical leaves is the Levi-Civita connection of the restriction of $\gamma$ to the leaves. We refer the reader to \cite{VL} for the curvature properties of the canonical connection.
\begin{rem}\label{obsnablaDgen} {\rm
The definition of Bott and canonical connection extends to arbitrary foliations on a (pseudo) Riemannian manifold and the characterization (\ref{propconexcan}) is correct in the general case \cite{{V71},{V73}}.}\end{rem}

Formula (\ref{tCartan}) with $\nabla^0$ replaced by $\nabla^B$ and with vertical arguments $\mathcal{Y}$ yields a vertical Cartan tensor $C$ associated with the tangent metric $\gamma$. In tangent bundle geometry, usually, it is the horizontal reflection $C^S$ that is called the (horizontal) Cartan tensor. A local calculation that uses the bases (\ref{Hbases}) yields the formula
$$ C^S(\mathcal{X},\mathcal{X}',\mathcal{X}'')=\nabla^D_{S\mathcal{X}}g (\mathcal{X}',\mathcal{X}''),\;\; \mathcal{X},\mathcal{X}',\mathcal{X}''\in\mathcal{H},$$
where the arguments $\mathcal{X}$ are horizontal and $(\mathcal{H},g)$ is the pair associated to $\gamma$.

A tangent metric $\gamma$ is called a {\it Lagrange metric} if the corresponding tensor $g$ is given by
$g_{ij}=\partial^2\mathcal{L}/\partial y^i\partial y^j$,
where the (continuous and smooth outside the zero section) function $\mathcal{L}$ on $\mathcal{T}M$ is a {\it regular Lagrangian} (regularity means that $g$ is non-degenerate).
The pair $(M,\mathcal{L})$ is called a {\it Lagrange manifold} \cite{BM}.
{\it Finsler metrics} are Lagrange metrics with a Lagrangian of the form $\mathcal{L}=\mathcal{F}^2$, where $\mathcal{F}$ is positive and positive homogeneous of degree $1$ and the corresponding Lagrange metric is positive definite \cite{BCS}. Then, $(M,\mathcal{F})$ is a {\it Finsler manifold}. The functions $\mathcal{L},\mathcal{F}$ are also called a Lagrange and Finsler metric, respectively.

The tensor $C^S$ is totally symmetric iff $\gamma$ is a {\it locally Lagrange metric}, i.e., each point has a neighborhood where $\gamma$ is a Lagrange metric. But, such a metric $\gamma$ is a globally Lagrange metric iff some cohomological obstructions vanish \cite{VL}.
\begin{rem}\label{obsconexChern} {\rm A regular Lagrangian $\mathcal{L}$ defines a canonical horizontal bundle $\mathcal{H}_{\mathcal{L}}$ called the {\it Cartan nonlinear connection} \cite{{BCS},{BM}} and there exists a canonical tangent metric
$\gamma_{\mathcal{L}}$ associated with the pair $(\mathcal{H}_{{\mathcal{L}}},g)$.
The restriction $\nabla^C$ of the canonical connection of a tangent metric $\gamma$ to $\mathcal{H}$ may be called the {\it Chern connection}, because it coincides with the {\it Rund-Chern connection} in the case of Finsler manifold. This is shown by a comparison of the connection coefficients given by formula (2.20) of \cite{V71} and (2.4.10) of \cite{BCS}.
The $S$-reflection $\nabla^{H}$ of the restriction of the canonical connection of $\gamma$ to $\mathcal{V}$ will be called the {\it Hashiguchi connection}, again, because it yields the connection bearing this name in Finsler geometry, as shown by a comparison of the connection coefficients calculated by formulas (2.20), (2.21) of \cite{V71} and Theorem 5.6.4 of \cite{BM}. Furthermore, in the Lagrange and Finsler case, the restriction of the Berwald connection $\nabla^B$ to $\mathcal{H}$ is the Berwald connection of Finsler and Lagrange geometry.}\end{rem}
\section{Lagrange-Hessian geometry}
Formula (\ref{coordtr}) shows that the vertical leaves of a tangent manifold $\mathcal{T}M$ are affine manifolds with affine coordinates $(y^i)$ and a (locally) Lagrange metric produces  (pseudo) Hessian metrics of the vertical leaves, which differentiably depend on the ``parameters" $x^i$. We refer to the geometry of this {\it leaf-wise (pseudo) Hessian metric} as {\it Lagrange-Hessian geometry}.

For instance, we will consider the {\it Lagrange-Hessian curvature} as follows.
A vector field $\mathcal{Y}\in\mathcal{V}$ is parallel on the leaves iff $S'\mathcal{Y}$ is projectable to $M$. It follows that the second formula (\ref{Berwald}) is equivalent to the fact that $\nabla^B_{\mathcal{Y}}\mathcal{Y}'=0$ for vertical, parallel vector fields $\mathcal{Y}'$, and we see that $\nabla^B$ is flat along the vertical leaves. $\nabla^B|_{\mathcal{V}}$ is also torsionless since $[Y,Y']=0$ for any vertical, parallel vector fields $Y,Y'$. Hence, $\nabla^B|_{\mathcal{V}}$ is the connection with the role of $\nabla^0$ of Section 1. On the other hand, if $\gamma$ is a tangent metric of $\mathcal{T}M$ with the corresponding horizontal bundle $\mathcal{H}=\mathcal{V}^{\perp_\gamma}$ and the corresponding transversal metric $g$, the Levi-Civita connection of the vertical leaves is the restriction of the canonical connection $\nabla^D$ of $\gamma$. Therefore, $\nabla^D$ plays the role of the connection $\nabla$ of Section 1 and we have a {\it difference tensor}
$$ \Phi(\mathcal{Z},\mathcal{Z}') =\nabla^D_{\mathcal{Z}}\mathcal{Z}'-\nabla^B_{\mathcal{Z}}\mathcal{Z}',
\;\mathcal{Z},\mathcal{Z}'\in T\mathcal{T}M.
$$
In particular, $\Phi(\mathcal{X},\mathcal{Y})=0,\;\Phi(\mathcal{Y},\mathcal{X})=0$, if $\mathcal{X}\in\mathcal{H},\mathcal{Y}\in\mathcal{V}$, and properties (\ref{propPhi}) with $\nabla$ replaced by $\nabla^D$ hold.

Then, formulas (\ref{eqHessc}) with vertical, parallel arguments $\mathcal{Y}$ define the notion of Lagrange-Hessian curvature $Q$ of a tangent, in particular a Lagrange, metric.
\begin{example}\label{excurbct} {\rm Assume that the tangent metric $\gamma$ is projectable, i.e., the horizontal, tensorial components $g_{ij}$ of $\gamma$ depend only on $x$. This implies that we are in the Lagrange case, namely, $g_{ij}=(1/2)(\partial^2(g_{ij}y^iy^j)/ \partial y^i\partial y^j)$ and that the Cartan tensor $C$ vanishes. Since $g_{ij}$ are constant along the vertical leaves, the Christoffel symbols of each leaf vanish and formulas (\ref{exprQgotic}), (\ref{Xi}) and (\ref{tCartan1}) (which hold in the present case too since $x^i$ are just ``parameters") imply the vanishing of the Lagrange-Hessian curvature.}\end{example}
\begin{rem}\label{relHH} {\rm The Lagrange-Hessian curvature defines the Hashiguchi curvature operator in vertical directions by the formula
$$\gamma(R_{\nabla^H}(S\mathcal{X}_1,S\mathcal{X}_2)\mathcal{X},\mathcal{X}')
=\frac{1}{2}[Q^S(\mathcal{X}',\mathcal{X}_1, \mathcal{X}_2,\mathcal{X})
-Q^S(\mathcal{X}',\mathcal{X}_2,\mathcal{X}_1,\mathcal{X})],
$$
where $\mathcal{X}$ with and without indices are horizontal vectors and the upper index $S$ denotes reflection.
Indeed, the definition of the Hashiguchi connection implies
$$R_{\nabla^H}(\mathcal{Z}_1,\mathcal{Z}_2)\mathcal{X}=
S'R_{\nabla^D}(\mathcal{Z}_1,\mathcal{Z}_2)(S\mathcal{X})\hspace{3mm} (\mathcal{Z}_1,\mathcal{Z}_2\in T\mathcal{T}M),$$
whence,
$$\gamma(R_{\nabla^H}(S\mathcal{X}_1,S\mathcal{X}_2)\mathcal{X},\mathcal{X}')
=R_{\nabla^D}(S\mathcal{X}',S\mathcal{X},S\mathcal{X}_1,S\mathcal{X}_2).$$
Thus, formula (\ref{relRQ}) implies the required relation.}\end{rem}
\begin{rem}\label{obsargoriz} {\rm For horizontal arguments, assumed to be projectable vector fields, the torsion terms of the second formula (\ref{curbDB}) for $\nabla^D,\nabla^B$ vanish and the curvature is
$$ \begin{array}{r}
R_{\nabla^D}(\mathcal{X}_1,\mathcal{X}_2, \mathcal{X}_3,\mathcal{X}_4)=R_{\nabla^B}(\mathcal{X}_1,\mathcal{X}_2, \mathcal{X}_3,\mathcal{X}_4)\vspace*{2mm}\\	 +\mathfrak{Q}(\mathcal{X}_1,\mathcal{X}_3, \mathcal{X}_4,\mathcal{X}_2)-\mathfrak{Q}(\mathcal{X}_1,\mathcal{X}_4, \mathcal{X}_3,\mathcal{X}_2).\end{array}$$}\end{rem}

Now, we address the subject of the correspondence Hessian versus K\"ahler and we will show how to extend the correspondence between a Hessian metric on the locally affine manifold $N$ and a K\"ahler metric on the tangent manifold $\mathcal{T}N$ \cite{S} to Lagrange-Hessian metrics of a Lagrange space $M^m$. There is no need to recall the original construction of \cite{S} because it amounts to the case of an isolated leaf in the general construction.

We consider the total space $\mathcal{T}(\mathcal{V})$ of the tangent bundle of the vertical leaves. This is a $3m$-dimensional manifold, which we will call the {\it vertical tangent manifold} of $M$. The iterated tangent manifold $\mathcal{T}(\mathcal{T}M)$ has local coordinates $(x^i,y^i,\xi^i,\eta^i)$, where $x,y$ change by formulas (\ref{coordtr}) and $\xi,\eta$ are vector coordinates with respect to the bases $(\partial/\partial x^i,\partial/\partial y^i)$ and change as follows
\begin{equation}\label{coordtr2} \tilde{\xi}^i=\frac{\partial\tilde{x}^i}{\partial x^j}\xi^j,
\,\tilde{\eta}^i=\frac{\partial\tilde{y}^i}{\partial x^j}\xi^j
+\frac{\partial\tilde{x}^i}{\partial x^j}\eta^j.
\end{equation}
The vertical tangent manifold $\mathcal{T}(\mathcal{V})$ is the submanifold
of $\mathcal{T}(\mathcal{T}M)$ defined by $\xi^i=0$. On the other hand, $\mathcal{T}M$ may be identified with the submanifold $\mathcal{T}(\mathcal{V})$ defined by $\eta^i=0$, which is the zero section of the projection $q: \mathcal{T}(\mathcal{V})\rightarrow\mathcal{T}M$.

Formulas (\ref{coordtr2}) show that the projection
$p:\mathcal{T}(\mathcal{V})\rightarrow M$ given by $p(x,y,\eta)=x$ is the Whitney sum $\mathfrak{V}=\mathcal{V}_1\oplus\mathcal{V}_2$, where
$\mathcal{V}_2\approx\mathcal{V}_1=\mathcal{V}$. $\mathfrak{V}$ will be called the {\it double vertical bundle (foliation)} and we will denote $\iota_a:\mathcal{V}_a\rightarrow\mathfrak{V}$ $(a=1,2)$ the identification of $\mathcal{V}$ with the two terms of $\mathfrak{V}$. Notice also the {\it flip involution} $\phi:\mathcal{T}(\mathcal{V})\rightarrow\mathcal{T}(\mathcal{V})$ defined by $\phi(x,y,\eta)=(x,\eta,y)$.

Another way of looking at the manifold $\mathcal{T}(\mathcal{V})$ is to identify it with the total space of the complexified tangent bundle $T^cM=TM\otimes\mathds{C}$ such that $\mathcal{V}_1$ is the real part and $\mathcal{V}_2$ is the imaginary part of the complexification. This interpretation shows that a horizontal bundle $\mathcal{H}$ on $\mathcal{T}M$ may also be seen as a horizontal bundle on $\mathcal{T}(\mathcal{V})$, i.e., $T\mathcal{T}(\mathcal{V})=\mathcal{H}\oplus\mathfrak{V}$.

Locally, on $\mathcal{T}(\mathcal{V})$ we have tangent bases $(X_i,\partial/\partial y^i,\partial/\partial\eta^i)$, where $X_i$ is given by (\ref{Hbases}), and dual bases
\begin{equation}\label{cobaseTV}
\theta^i=dy^i+t^i_jdx^j,\,\kappa^i=d\eta^i+t^i_jdx^j
 \end{equation} with the same coefficients $t^i_j$.

We recall that a CR structure is a complex tangent distribution $E$ that is involutive and such that $E\cap\bar{E}=0$ (the bar denotes complex conjugation). On the other hand, a tangent bundle endomorphism $F$ such that $F^3+F=0$ is an F structure. Then,  $F$ has the eigenvalues $\pm i,0$ and, if the $i$-eigenbundle $E$ is involutive, $F$ is a CRF structure \cite{VCRF}.
\begin{prop}\label{strCRF} For any choice of a horizontal bundle, there exists a canonical CRF structure $\mathfrak{J}$ on the vertical tangent manifold $\mathcal{T}(\mathcal{V})$.\end{prop}
\begin{proof} On $\mathfrak{V}\approx p^{-1}(T^cM)$, multiplication by $i$ defines a complex bundle structure, which provides a complex structure $\mathfrak{J}_{\mathfrak{V}}$ along the {\it double vertical leaves} with the local expression
 \begin{equation}\label{JVlocal}
\mathfrak{J}_{\mathfrak{V}}\frac{\partial}{\partial y^i}=\frac{\partial}{\partial\eta^i},\;
\mathfrak{J}_{\mathfrak{V}}\frac{\partial}{\partial\eta^i} =
-\frac{\partial}{\partial y^i}.\end{equation}The local expression shows the integrability of $\mathfrak{J}_{\mathfrak{V}}$ along the leaves.
We get the required tensor $\mathfrak{J}$ by putting $\mathfrak{J}|_{\mathfrak{V}}=\mathfrak{J}_{\mathfrak{V}}, \mathfrak{J}|_{\mathcal{H}}=0$. \end{proof}

The leaf-wise complex structure $\mathfrak{J}_{\mathfrak{V}}$ may also be defined by means of the endomorphisms $S_a$ $(a=1,2)$ defined on $T\mathcal{T}(\mathcal{V})$ by the formula $S_a(\mathfrak{X})=(\iota_a\circ S)(p_*\mathfrak{X})^h$, where $\mathfrak{X}\in T\mathcal{T}(\mathcal{V})$ and $S$ is the tangent structure of $\mathcal{T}M$. The tensor fields $S_a$ are Nijenhuis tensors of a constant rank such that $S_a^2=0$. The vanishing of $\mathcal{N}_{S_a}$ follows from the local expressions
$$S_1\frac{\partial}{\partial x^i}=\frac{\partial}{\partial y^i}, S_2\frac{\partial}{\partial x^i}=\frac{\partial}{\partial \eta^i}.$$
The structure $\mathfrak{J}_{\mathfrak{V}}$ is determined by the equalities
$$ \mathfrak{J}_{\mathfrak{V}}\circ S_1=S_2,\;\mathfrak{J}_{\mathfrak{V}}\circ S_2=-S_1.$$

We shall need metrics that are the analog of tangent metrics and it is convenient to define them using the tensors $S_a$.
\begin{defin}\label{defadaptmetric} {\rm A (pseudo) Riemannian metric $\mathfrak{g}$ on $\mathcal{T}(\mathcal{V})$ will be a {\it double tangent metric} if $\mathcal{V}_1\perp_{\mathfrak{g}}\mathcal{V}_2$, $\mathfrak{g}|_{\mathcal{V}_2}$is non degenerate and
$$\mathfrak{g}(S_1\mathcal{X},S_1\mathcal{X}')= \mathfrak{g}(S_2\mathcal{X},S_2\mathcal{X}')= \mathfrak{g}(\mathcal{X},\mathcal{X}'),\; \forall\mathcal{X},\mathcal{X}' \in\mathcal{H}\perp_{\mathfrak{g}}\mathfrak{V}.$$}\end{defin}

With the cotangent bases (\ref{cobaseTV}), a double tangent metric may be written as
\begin{equation}\label{Jlocal}	\mathfrak{g}= g_{ij}dx^i\otimes dx^j+g_{ij}\theta^i\otimes\theta^j +g_{ij}\kappa^i\otimes\kappa^j.\end{equation}
If $\mathfrak{g}$ is a double tangent (pseudo) Riemannian metric, $\mathfrak{g}|_{\mathcal{V}_1}$ and $\mathfrak{g}|_{\mathfrak{V}}$ are non degenerate and $\phi$ is an isometry. For a function $\mathcal{L}$ on $\mathcal{T}(\mathcal{V})$, the horizontal tensor $(\partial^2\mathcal{L}/\partial y^i\partial y^j)dx^i\otimes dx^j$ is still invariant, and non degenerate in the {\it regular} case. Accordingly, we may extend the notions of {\it locally Lagrange}, {\it Lagrange} and {\it Finsler} to double tangent metrics.

It follows easily that any double tangent metric $\mathfrak{g}$ is compatible with the CRF structure tensor $\mathfrak{J}$ in the sense that
$$ \mathfrak{g}(\mathfrak{J}\mathfrak{Z},\mathfrak{Z}') +\mathfrak{g}(\mathfrak{Z},\mathfrak{J}\mathfrak{Z}')=0,
\;\forall\mathfrak{Z},\mathfrak{Z}'\in T\mathcal{T}(\mathcal{V})$$ and, with the terminology of \cite{VCRF}, $(\mathfrak{J},\mathfrak{g})$ is a {\it metric CRF structure}. This implies that the restriction of $\mathfrak{g}$ to the leaves of $\mathfrak{V}$ are Hermitian for the complex structure $\mathfrak{J}_{\mathfrak{V}}$.

Clearly, the double tangent metrics are in a bijective correspondence with pairs $(\mathcal{H},g)$ where $\mathcal{H}$ is a horizontal bundle on $\mathcal{T}(\mathcal{V})$ and $g$ is a non degenerate metric on $\mathcal{H}$.
A tangent metric $\gamma$ on $\mathcal{T}M$ defines a horizontal bundle $\mathcal{H}$ endowed with a metric $g$ and the interpretation of $\mathcal{T}(\mathcal{V})$ by means of $T^cM$ allows the identification of $(\mathcal{H},g)$ with a similar pair on $\mathcal{T}(\mathcal{V})$. Accordingly, we get a double tangent metric $\mathfrak{g}_{\gamma}$ on $\mathcal{T}(\mathcal{V})$ called the {\it extension of $\gamma$}.
If the two first terms of (\ref{Jlocal}) express the tangent metric $\gamma$ on $\mathcal{T}M$, (\ref{Jlocal}) is the extension of the former to $\mathcal{T}(\mathcal{V})$.

The next proposition shows the correspondence between locally Lagrange metrics on $\mathcal{T}M$ and the (pseudo) K\"ahler metrics on $\mathcal{T}(\mathcal{V})$.
\begin{prop}\label{propKahler} Let $\gamma$ be a tangent metric on $\mathcal{T}M$ and $\mathfrak{g}_{\gamma}$ its extension to $\mathcal{T}(\mathcal{V})$. Then, the restriction of $\mathfrak{g}_{\gamma}$ to the leaves of $\mathfrak{V}$ is a (pseudo) K\"ahler metric iff $\gamma$ is locally Lagrange.\end{prop}
\begin{proof} Formula (\ref{JVlocal}) shows that $z^i=y^i+\sqrt{-1}\eta^i$ are holomorphic coordinates along the leaves $x^i=const.$. Then, from (\ref{Jlocal}) we see that the metric induced by $\mathfrak{g}_{\gamma}$ on these leaves is given by
$$\mathfrak{g}_{\gamma}|_{\mathfrak{V}}=g_{ij}dz^i\otimes d\bar{z}^j$$ and the corresponding K\"ahler form is
$$\omega=\frac{i}{2}g_{ij}dz^i\wedge d\bar{z}^j=g_{ij}dy^i\wedge d\eta^j,$$
where $g_{ij}$ are the horizontal components of the given metric $\gamma$.
Since $g_{ij}=g_{ij}(x,y)$, it follows that $d\omega=0$ along the leaves $x=const.$ iff
$\partial g_{ij}/\partial y^k=\partial g_{kj}/\partial y^i$, i.e., iff the Cartan tensor of $\gamma$ is symmetric, therefore, $\gamma$ is a locally Lagrange metric.\end{proof}

In order to get a corresponding relationship between the Lagrange-Hessian and K\"ahler-Riemannian curvatures we need an adequate connection, which is provided by the following proposition.
\begin{prop}\label{propconexherm} Let $\mathfrak{g}$ be a double tangent metric on $\mathcal{T}(\mathcal{V})$. Then, there exists a unique Bott connection $\mathfrak{D}$, with respect to the foliation $\mathfrak{V}$, that has the following properties:\\
1) $(\mathfrak{D}_{\mathfrak{X}}\mathfrak{g})(\mathfrak{X}',\mathfrak{X}'')=0$,
$(\mathfrak{D}_{\mathfrak{Y}}\mathfrak{g})(\mathfrak{Y}',\mathfrak{Y}'')=0$,
\\ \noindent
2) $pr_{\mathcal{H}}T_{\mathfrak{D}}(\mathfrak{X},\mathfrak{X}')=0$,\\ \noindent
3) $\mathfrak{D}_{\mathfrak{Y}}(\mathfrak{J}\mathfrak{Y}')=
\mathfrak{J}(\mathfrak{D}_{\mathfrak{Y}}\mathfrak{Y}')$,\\ \noindent
4) $T_{\mathfrak{D}}(\mathfrak{J}\mathfrak{Y},\mathfrak{Y}')
=T_{\mathfrak{D}}(\mathfrak{Y},\mathfrak{J}\mathfrak{Y}')$,\\
\noindent
where $\mathfrak{X},\mathfrak{X}',\mathfrak{X}''\in\mathcal{H}$, $\mathfrak{Y},\mathfrak{Y}',\mathfrak{Y}''\in\mathfrak{V}$.
\end{prop}
\begin{proof}
Let us extend the field of scalars to $\mathds{C}$ and define the Hermitian metric $\mathfrak{g}^c$ on $T^c\mathcal{T}(\mathcal{V})$ by
$$\mathfrak{g}^c(i\mathfrak{Z},\mathfrak{Z}') =i\mathfrak{g}(\mathfrak{Z},\mathfrak{Z}'),
\;\mathfrak{g}^c(\mathfrak{Z},i\mathfrak{Z}') =-i\mathfrak{g}(\mathfrak{Z},\mathfrak{Z}'),$$
where $\mathfrak{Z},\mathfrak{Z}'$ are real vectors. We also extend connections $\mathfrak{D}$ to complex vector fields by requiring complex linearity. Being a Bott connection, the required $\mathfrak{D}$ preserves $\mathcal{H},\mathfrak{V}$ and
$$\mathfrak{D}_{\mathfrak{X}}\mathfrak{Y}=pr_{\mathfrak{V}}[\mathfrak{X},\mathfrak{Y}] \,
\mathfrak{D}_{\mathfrak{Y}}\mathfrak{X}=pr_{\mathcal{H}}[\mathfrak{Y},\mathfrak{X}] \hspace{2mm}
(\mathfrak{X}\in\mathcal{H},\mathfrak{Y}\in\mathfrak{V}).$$
By property 3), $\mathfrak{D}$ also preserves the eigenbundles $E,\bar{E}$ of $\mathfrak{J}_{\mathfrak{V}}$ and, if $\mathfrak{Y}\in E,\mathfrak{Y}'\in\bar{E}$, property 4) yields
$$\mathfrak{D}_{\mathfrak{Y}}\mathfrak{Y}'=pr_{\bar{E}}[\mathfrak{Y},\mathfrak{Y}'],\;
\mathfrak{D}_{\mathfrak{Y}'}\mathfrak{Y}=pr_{E}[\mathfrak{Y},\mathfrak{Y}'].$$
The covariant derivatives $\mathfrak{D}_{\mathfrak{X}}\mathfrak{X}'$, $\mathfrak{X},\mathfrak{X}'\in\mathcal{H}$ can be obtained from the first condition 1) and condition 2) like in the well known case of a Riemannian connection \cite{KN}, Proposition IV.2.3. Finally, in order to get the covariant derivatives
$\mathfrak{D}_{\mathfrak{Y}}\mathfrak{Y}'$ where $\mathfrak{Y},\mathfrak{Y}'$ belong both either to the $i$ or the $-i$-eigenbundle, we notice that the second condition 1) is equivalent to
\begin{equation}\label{EEE}
\mathfrak{Y}''(\mathfrak{g}^c(\mathfrak{Y},\mathfrak{Y}')
-\mathfrak{g}^c(\mathfrak{D}_{\mathfrak{Y}''}\mathfrak{Y},\mathfrak{Y}')
-\mathfrak{g}^c(\mathfrak{D}_{\mathfrak{Y}''}\overline{\mathfrak{Y}'}, \overline{\mathfrak{Y}})=0,
\end{equation}
where bar denotes complex conjugation. If all the arguments belong to either $E$ or $\bar{E}$, we already have the covariant derivatives $\mathfrak{D}_{\mathfrak{Y}''}\overline{\mathfrak{Y}'}$ and the equality (\ref{EEE}) determines $\mathfrak{D}_{\mathfrak{Y}''}\mathfrak{Y}$.
The obtained results also show that $\mathfrak{D}$ is the complexification of a real connection.\end{proof}

We will say that $\mathfrak{D}$ is the {\it Hermitian connection} of $\mathfrak{g}$ since, along the leaves of $\mathfrak{V}$, $\mathfrak{D}$ is the Hermitian connection of the leaves (\cite{KN}, Proposition IX.10.2 and \cite{V73} Theorem 4.6.8). If the metric $\mathfrak{g}$ is the extension of a locally Lagrange, tangent metric $\gamma$ of $\mathcal{T}M$, then, by Proposition \ref{propKahler}, $\mathfrak{g}$ restricts to K\"ahler metrics on the leaves of $\mathfrak{V}$ and the Hermitian connection of the leaves coincides with the Riemannian connection (\cite{KN}, Theorem IX.4.3). This implies that the Hermitian connection $\mathfrak{D}$ satisfies the properties (\ref{propconexcan}) with the foliation $\mathcal{V}$ replaced by $\mathfrak{V}$. Hence, $\mathfrak{D}$ is the canonical connection $\tilde{\nabla}^D$ of the pair $(\mathfrak{V},\mathfrak{g})$ (see Remark \ref{obsnablaDgen}), $D$ being the Levi-Civita connection of $\mathfrak{g}$. This observation leads to the following result.
\begin{prop}\label{curbpeTrond} Let $\mathfrak{g}$ be the extension of the locally Lagrange, tangent metric $\gamma$ of $\mathcal{T}M$ and $\nabla^D$ the canonical connection of $\gamma$, then, along $\mathcal{T}M$ seen as a submanifold of $\mathcal{T}(\mathcal{V})$, one has
\begin{equation}\label{curburiegale}
R_{\tilde{\nabla}^D}(\mathfrak{Z}_1,\mathfrak{Z}_2,
\mathfrak{Z}_3,\mathfrak{Z}_4)=
R_{\nabla^D}(\mathfrak{Z}_1,\mathfrak{Z}_2,\mathfrak{Z}_3,\mathfrak{Z}_4),
\;\forall\mathfrak{Z}_1,\mathfrak{Z}_2,\mathfrak{Z}_3,\mathfrak{Z}_4
\in T\mathcal{T}M.\end{equation}\end{prop}
\begin{proof} The submanifold $\mathcal{T}M$ has the local equations $\eta^i=0$.
As previously noticed, we have $\tilde{\nabla}^D=\mathfrak{D}$. Since $\mathfrak{D}$ is a real connection that commutes with $\mathfrak{J}$, the isomorphism of complex vector bundles $(\mathfrak{V},\mathfrak{J}_{\mathfrak{V}})\approx\mathcal{V}^c$ tells that $\tilde{\nabla}^D$ must preserve the real part $\mathcal{V}$ of the complexification $\mathcal{V}^c$. On the other hand, the connection induced by $\tilde{\nabla}^D$ in the subbundle $T\mathcal{T}M\subset T\mathcal{T}(\mathcal{V})|_{\mathcal{T}M}$ also satisfies (\ref{propconexcan}), therefore, it just is the canonical connection $\nabla^D$ on $\mathcal{T}M$. Accordingly, the covariant derivatives in the two curvature tensors of (\ref{curburiegale}) are the same and we are done.\end{proof}

Furthermore, we define the {\it extended Berwald connection} $\tilde{\nabla}^B$ on $\mathcal{T}(\mathcal{V})$ to be the Bott connection with respect to the foliation $\mathfrak{V}$ and the horizontal bundle $\mathcal{H}$ such that
\begin{equation}\label{Brerwaldlarg} \tilde{\nabla}^B_{\mathfrak{X}}\mathfrak{X}' =S'_1pr_{\mathcal{V}_1}[\mathfrak{X},S_1\mathfrak{X}'],
\;\tilde{\nabla}^B_{\mathfrak{Z}}\mathfrak{Z}'=0,\end{equation}
$\forall\mathfrak{X},\mathfrak{X}'\in\mathcal{H}$, $\forall\mathfrak{Z}\in\mathfrak{V}$ and all parallel fields $\mathfrak{Z}'\in\mathfrak{V}$ (i.e., $\mathfrak{Z}'=
\lambda_i(x)(\partial/\partial y^i)+ \zeta^i(x)(\partial/\partial\eta^i$), where
$S'_1$ is defined by $(S_1|_{\mathcal{H}})^{-1}$ on $\mathcal{V}_1$ and by zero on $\mathcal{H}$ and $\mathcal{V}_2$. The other covariant derivatives $\tilde{\nabla}^B$ are provided by the Bott condition (\ref{BottV}) with $\mathcal{X},\mathcal{Y}$ replaced by $\mathfrak{X},\mathfrak{Z}$, respectively. In particular, using a projectable field $$\mathfrak{X}=\xi^i(x)X_i=\xi^i(x)(\frac{\partial}{\partial x^i}-
t_i^j(x,y)\frac{\partial}{\partial y^j}),$$ we see that
$\tilde{\nabla}^B_{\mathfrak{X}}$ preserves $\mathcal{V}_1,\mathcal{V}_2$, separately. The second part of (\ref{Brerwaldlarg}) shows that the same is true for $\tilde{\nabla}^B_{\mathfrak{Z}}$. Furthermore, using the definition of $\mathfrak{J}$ and the local formulas (\ref{JVlocal}), we see that $\tilde{\nabla}^B\mathfrak{J}=0$. By comparing the definitions, we also see that $\tilde{\nabla}^B$ induces the Berwald connection $\nabla^B$ on the submanifold $\mathcal{T}M\subset\mathcal{T}(\mathcal{V})$.

Finally, we can prove the following proposition which is the announced relation between curvatures.
\begin{prop}\label{R1/2Q} Let $\mathfrak{g}$ be the extension of the locally Lagrange, tangent metric $\gamma$ of $\mathcal{T}M$. Then, at any point of the submanifold $\mathcal{T}M$ and for any arguments $\mathcal{Y}_a\in\mathcal{V}_1$ $(a=1,2,3,4)$, the following relation holds
\begin{equation}\label{eqR1/2Q}
R_{\tilde{\nabla}^D}(\mathcal{Y}_1, \mathcal{Y}_2,\mathcal{Y}_3,\mathcal{Y}_4)=
\frac{1}{2}Q(\mathcal{Y}_1, \mathcal{Y}_2,\mathcal{Y}_3,\mathcal{Y}_4).
\end{equation}\end{prop}
\begin{proof}
In the interpretation of $\mathcal{T}(\mathcal{V})$ as the total space of the complexified tangent bundle $T^cM$, $\tilde{\nabla}^D,\tilde{\nabla}^B,\tilde{\Phi}= \tilde{\nabla}^D-\tilde{\nabla}^B$ and $\tilde{Q}$ defined by taking complex arguments $\mathcal{Y}$ in formula (\ref{eqHessc}) are the extension of
$\nabla^D,\nabla^B,\Phi,Q$ to complex arguments by $\mathds{C}$-linearity.
Accordingly, if transposed to complex arguments the proof of formula (\ref{relRQ}) holds, which means that we have
\begin{equation}\label{auxcurb1}
R_{\nabla^D}(\mathfrak{Y}',\mathfrak{Y},\mathfrak{Y}_1,\mathfrak{Y}_2)
=\frac{1}{2}[Q(\mathfrak{Y}',\mathfrak{Y}_1,\mathfrak{Y}_2,\mathfrak{Y})
-Q(\mathfrak{Y}',\mathfrak{Y}_2,\mathfrak{Y}_1,\mathfrak)],\end{equation}
where the arguments are parallel vector fields in $\mathfrak{V}$.
A known property of the curvature tensor of a K\"ahler metric tells us that we have
$$ R_{\tilde{\nabla}^D}(\mathfrak{J}\mathcal{Y}_1, \mathfrak{J}\mathcal{Y}_2,\mathcal{Y}_3,\mathcal{Y}_4)=	 R_{\tilde{\nabla}^D}(\mathcal{Y}_1, \mathcal{Y}_2,\mathcal{Y}_3,\mathcal{Y}_4).$$
Now, for parallel arguments $\mathcal{Y}\in\mathcal{V}_1$, (\ref{auxcurb1}) becomes
$$R_{\tilde{\nabla}^D}(\mathfrak{J}\mathcal{Y}_1, \mathfrak{J}\mathcal{Y}_2,\mathcal{Y}_3,\mathcal{Y}_4)=
\frac{1}{2}[\mathfrak{g}(\mathfrak{J}\mathcal{Y}_1, \tilde{Q}(\mathcal{Y}_3,\mathcal{Y}_4)(\mathfrak{J}\mathcal{Y}_2))
-\mathfrak{g}(\mathfrak{J}\mathcal{Y}_1,\tilde{Q}(\mathfrak{J}\mathcal{Y}_3, \mathcal{Y}_4) \mathcal{Y}_2)],$$
where the last term vanishes because it is a scalar product of orthogonal vectors. For the first term, the interpretation of $\tilde{Q}$ as the complexification	of $Q$ yields $$\mathfrak{g}(\mathfrak{J}\mathcal{Y}_1, \tilde{Q}(\mathcal{Y}_3,\mathcal{Y}_4)(\mathfrak{J}\mathcal{Y}_2))=
\mathfrak{g}(\mathcal{Y}_1, \tilde{Q}(\mathcal{Y}_3,\mathcal{Y}_4)(\mathcal{Y}_2))= Q(\mathcal{Y}_1,\mathcal{Y}_2,\mathcal{Y}_3,\mathcal{Y}_4).$$ Combining the results we get the required conclusion.

If we take the arguments of (\ref{eqR1/2Q}) in the basis $\partial/\partial y^i$ and decompose them into the sum of the holomorphic and anti-holomorphic part, then, using the properties of the curvature tensor of a K\"ahler metric, we will get Proposition 3.3 of \cite{S}.
\end{proof}
\section{Appendix}
In this appendix we give an index-free presentation of some more facts concerning Hessian curvature on locally affine manifolds that were treated via local coordinates in \cite{S}. The notation is the same as in Section 1.2.

The importance of the symmetry properties of the Riemannian curvature tensor suggests looking for symmetry properties of the Hessian curvature. These follow from the following result that is equivalent to formula (1) of Proposition 3.1 of \cite{S}.
\begin{prop}\label{3.1.S}  The value of the Hessian curvature tensor on parallel arguments is given by the formula
\begin{equation}\label{QcuC}
Q(\mathcal{Y}_1,\mathcal{Y}_2,\mathcal{Y}_3,\mathcal{Y}_4)=
\mathcal{Y}_2(C(\mathcal{Y}_1,\mathcal{Y}_3,\mathcal{Y}_4)) -2\gamma(\nabla_{\mathcal{Y}_2}\mathcal{Y}_1, \nabla_{\mathcal{Y}_3}\mathcal{Y}_4).\end{equation}\end{prop}
\begin{proof} From (\ref{eqHessc}) and the definition of $\Phi$ we get
$$\begin{array}{l}
Q(\mathcal{Y}_1,\mathcal{Y}_2,\mathcal{Y}_3,\mathcal{Y}_4)=
\gamma(\nabla^0_{\mathcal{Y}_2}\nabla_{\mathcal{Y}_3}\mathcal{Y}_4, \mathcal{Y}_1)=
\gamma(\nabla_{\mathcal{Y}_2}\nabla_{\mathcal{Y}_3}\mathcal{Y}_4,  \mathcal{Y}_1)\vspace*{2mm}\\ -\gamma(\Phi(\mathcal{Y}_2,\nabla_{\mathcal{Y}_3}\mathcal{Y}_4), \mathcal{Y}_1)=
\gamma(\nabla_{\mathcal{Y}_2}\nabla_{\mathcal{Y}_3}\mathcal{Y}_4,  \mathcal{Y}_1)
-\frac{1}{2}C(\mathcal{Y}_1,\mathcal{Y}_2,\nabla_{\mathcal{Y}_3} \mathcal{Y}_4).
\end{array}$$
Then, using the total symmetry of $C$ and $\nabla\gamma=0$, we get the required result.\end{proof}
\begin{corol}\label{symQQ} The tensor field $Q$ has the following symmetry properties\footnote{The difference between these properties and those of Proposition 3.1 of \cite{S} is explained by our different choice of the order of arguments in the Hessian curvature.}
\begin{equation}\label{symQ1} \begin{array}{l}
Q(\mathcal{Y}_1,\mathcal{Y}_2,\mathcal{Y}_3,\mathcal{Y}_4)=
Q(\mathcal{Y}_1,\mathcal{Y}_2,\mathcal{Y}_4,\mathcal{Y}_3)\vspace*{2mm}\\
=Q(\mathcal{Y}_3,\mathcal{Y}_4,\mathcal{Y}_1,\mathcal{Y}_2)
=Q(\mathcal{Y}_2,\mathcal{Y}_1,\mathcal{Y}_3,\mathcal{Y}_4).\vspace*{2mm}\\
\end{array}\end{equation}
The same symmetries also hold for $\mathfrak{Q}(\mathcal{Y}_1,\mathcal{Y}_2,\mathcal{Y}_3,\mathcal{Y}_4)$.
\end{corol}
\begin{proof} Use the expressions (\ref{QcuC}) and (\ref{exprQgotic}).
\end{proof}
\begin{rem}\label{Duistermaat} {\rm An interesting consequence of formulas (\ref{QcuC}), (\ref{relRQ}) is that the Riemannian curvature of a Hessian metric involves only the first order derivatives of the metric, hence, only the third order derivatives of the function $\varphi$ whose Hessian defines the metric, while the usual expression of the Riemannian curvature contains the second order derivatives of the metric, hence, we should expect the fourth order derivatives of $\varphi$. This phenomenon was studied in \cite{Du}.}\end{rem}

A comparison with Riemannian geometry, again, suggests a notion of {\it Hessian sectional curvature} \cite{S}. As a matter of fact, the latter is an invariant associated to a quadratic cone of the tangent space of a locally affine manifold $N$ endowed with a metric $\gamma$.
In order to define it we need the following observation.
Any 4-times covariant tensor field $\Xi$ that has the symmetry properties (\ref{symQ1}) is equivalent to a quadratic form $\tilde{\Xi}$ on $\odot^2TN$, which is defined by
\begin{equation}\label{tildaXi}
\tilde{\Xi}(\mathcal{Y}\odot\mathcal{Y}',\mathcal{Y}\odot\mathcal{Y}')= \Xi(\mathcal{Y},\mathcal{Y}',\mathcal{Y},\mathcal{Y}')\end{equation}
on the generators $\mathcal{Y}\odot\mathcal{Y}'$.
Equivalently, with respect to the local basis $(\partial/\partial y^u)$, if $\tau$ is a symmetric, $2$-contravariant tensor, $\tilde{\Xi}(\tau,\tau)=\Xi_{uvst}\tau^{uv}\tau^{st}$.
\begin{defin}\label{defconicalcurv} {\rm Let $\nu$ be a quadratic cone defined by $\nu(\mathcal{Y},\mathcal{Y})=0$, where $\mathcal{Y}\in TN$ and $\nu$ is a $2$-covariant, symmetric, tensor with $||\nu||_\gamma\neq0$. Then,
the {\it conical} (sectional \cite{S}) {\it Hessian curvature} of $\nu$ is
$$\kappa(\nu)= \frac{\tilde{Q}(\sharp_\gamma\nu,\sharp_\gamma\nu)}{||\nu||^2_\gamma},$$
where $\tilde{Q}$ is defined by (\ref{tildaXi}).}\end{defin}

The value of $\kappa(\nu)$ does not change under the multiplication of $\nu$ by a scalar.
If $\mathcal{Y}\odot\mathcal{Y}'$ is a generator of $\odot^2TN$, the conical curvature of $\flat_\gamma(\mathcal{Y}\odot\mathcal{Y}')$ may be written (omitting $\flat_\gamma$) as
$$
\kappa(\mathcal{Y}\odot\mathcal{Y}') =\frac{Q(\mathcal{Y},\mathcal{Y}',\mathcal{Y},
\mathcal{Y}')}{G(\mathcal{Y},\mathcal{Y}',\mathcal{Y},\mathcal{Y}')},
$$
where
$$G(\mathcal{Y}_1,\mathcal{Y}_2,\mathcal{Y}_3,\mathcal{Y}_4)=
\gamma(\mathcal{Y}_1,\mathcal{Y}_3)\gamma(\mathcal{Y}_2,\mathcal{Y}_4)
+\gamma(\mathcal{Y}_1,\mathcal{Y}_4)\gamma(\mathcal{Y}_2,\mathcal{Y}_3)$$ has the same symmetry properties like $Q$.
\begin{prop}\label{curbconct} {\rm\cite{S}}	 The conical curvature of a (pseudo) Hessian metric $\gamma$ is independent of the cone, i.e., $\kappa(\nu)=f(y)\in C^\infty(\mathcal{T}M)$, iff
\begin{equation}\label{QfG}
Q(\mathcal{Y}_1,\mathcal{Y}_2,\mathcal{Y}_3,\mathcal{Y}_4)= fG(\mathcal{Y}_1,\mathcal{Y}_2,\mathcal{Y}_3,\mathcal{Y}_4).
\end{equation}
Moreover, in this case, and if $m\geq3$, $f=const.$\end{prop}
\begin{proof} If $\kappa=f$, the quadratic form $\widetilde{(Q-fG)}$ vanishes and so does the corresponding symmetric bilinear form. This fact exactly is (\ref{QfG}).
Furthermore, if (\ref{QfG}) holds, (\ref{relRQ}) implies
that the Riemannian, sectional curvature is $-(f/2)$ and Schur's theorem (Theorem V.2.2 in \cite{KN}) tells that, if $m\geq3$, $f=const.$
\end{proof}

Hessian metrics of constant conical curvature were studied in \cite{S} and, more recently in \cite{FK}. \vspace*{2mm}\\
{\bf Open problem.}
If the metric $\gamma$ is positive definite, so is the corresponding form $\tilde{G}$ given by (\ref{tildaXi}) and we can consider {\it principal cones} $C_i$ and {\it principal conical curvatures} $\lambda_i$, defined by the eigenvectors, respectively, the eigenvalues of $\tilde{Q}$ with respect to $\tilde{G}$. It would be interesting to study ``ombilical" Hessian manifolds defined by the equality of all the principal conical curvatures $\lambda_i$.

{\small Department of Mathematics, University of Haifa, Israel. E-mail: vaisman@math.haifa.ac.il}
\end{document}